\documentclass{amsart}

\usepackage{xcolor}
\usepackage{mathtools}
\usepackage{amsmath}
\usepackage{amsthm}
\usepackage{graphicx}
\usepackage{amssymb}
\usepackage{epstopdf}
\usepackage{nicefrac}
\usepackage{mathrsfs}
\usepackage{hyperref}
\usepackage{enumitem}
\usepackage{tikz}
\usepackage{caption}
\usepackage{subcaption}
\usepackage{cite}

\newtheorem{theorem}{Theorem}[section]
\newtheorem{lemma}[theorem]{Lemma}

\newtheorem{corollary}[theorem]{Corollary}
\theoremstyle{remark}
\newtheorem{remark}[theorem]{Remark}
\newtheorem{definition}[theorem]{Definition}
\newcommand{\HH}{\mathscr H}
\newcommand{\TT}{\mathbb T}
\newcommand{\DD}{\mathbb D}
\newcommand{\NN}{\mathbb N}

\newcommand{\PP}{\mathbb P}
\newcommand{\ind}{\mathbf 1}

\DeclareMathOperator{\supp}{supp}
\setlist[enumerate]{itemsep=3pt,topsep=5pt}

\title[A Fejér--Riesz inequality for Dirichlet series]{A Fejér--Riesz inequality for Dirichlet series}
\author{Karl-Mikael Perfekt}
\date{\small \today}

\begin{document}

\begin{abstract}
We prove the following inequality for Dirichlet polynomials:
\[
\int_0^1 |f(1/2+\sigma)|\,d\sigma \lesssim \lim_{T\to\infty} \frac{1}{2T} \int_{-T}^T |f(it)| \, dt.
\]
In particular, for a Dirichlet series $f(s) = \sum_{n\geq 1} a_n n^{-s}$ belonging to the Hardy space $\mathscr{H}^1$ of Dirichlet series, 
\[
 \left|a_1+\sum_{n=2}^\infty \frac{a_n}{\sqrt n\log n}\right|
 \lesssim \|f\|_{\HH^1}.
\]
This answers a question raised previously in the literature and it proves that the multiplicative Hilbert matrix has a bounded symbol.
\end{abstract}
\maketitle
\section{Introduction}

The classical Fejér--Riesz inequality states that for every analytic polynomial $f(z) = \sum_{n \leq N} \hat{f}(n) z^n$, 
\begin{equation} \label{eq:fejerriesz}
\int_{-1}^1 |f(x)|^p \, dx \leq \pi \int_0^{2\pi} |f(e^{i\theta})|^p \, \frac{d\theta}{2\pi}, \qquad p > 0.
\end{equation}
For $p = 2$, this inequality is equivalent to the boundedness of the Hilbert matrix $(1/(m+n+1))_{m,n=0}^\infty$, as is readily seen from the identity 
\[
\int_0^1 |f(x)|^2 \, dx = \sum_{m, n = 0}^\infty \frac{\hat{f}(m) \overline{\hat{f}(n)}}{m+n+1}.
\]

The normalized integral on the right-hand side of \eqref{eq:fejerriesz} is precisely the Hardy space $H^p(\mathbb{T})$-norm of $f$, and the inequality can thus be recast as a Carleson embedding statement:
\[
\int_{-1}^1 |f(x)|^p \, dx \leq \pi \|f\|_{H^p}^p, \qquad f \in H^p(\mathbb{T}).
\]
In one variable, the availability of the inner-outer factorization shows that the Carleson measure property is independent of $p$. Up to this fact, the truth of the Fejér--Riesz inequality for one $p = p_0$ is equivalent to its validity for all $p > 0$.

By contrast, in the setting of Hardy spaces of Dirichlet series $\HH^p$, $p \geq 1$, any  factorization mechanism of this type is known to be false \cite{OrtegaSeip}. Recall here that $\HH^p$-functions are represented by Dirichlet series, 
\[
f(s) = \sum_{n=1}^\infty a_n n^{-s}, \qquad \Re s > 1/2,
\]
and that the \textit{multiplicative Hilbert matrix} naturally arises from the identity
\[
   \int_0^\infty |f(1/2+\sigma)-a_1|^2 \,d\sigma = \sum_{m, n=2}^\infty \frac{a_m \overline{a_n}}{\sqrt {mn}\log mn},
\]
the right-hand side coinciding with its quadratic form \cite{BrevigEtAl}. A simple argument shows that this form is bounded, and thus that the multiplicative Hilbert matrix
\[ 
\left(\frac{1}{\sqrt{mn} \log mn} \right)_{m,n=2}^\infty \colon \ell^2(\mathbb{Z}_{\geq 2}) \to  \ell^2(\mathbb{Z}_{\geq 2})
\]
is a continuous operator.

However, as alluded to earlier, the transition to other values of $p$ is no longer straightforward. The most pertinent inequality is obtained when $p = 1$. Does it hold that
\begin{equation} \label{eq:signedhardy}
\left|\sum_{n=2}^\infty \frac{a_n}{\sqrt n\log n}\right| = \left| \int_0^\infty (f(1/2+\sigma) - a_1) \,d\sigma \right| \lesssim \|f\|_{\HH^1}?
\end{equation}
This problem was already identified in \cite{BrevigEtAl}, where it was noted that a positive resolution is equivalent to the existence of a bounded symbol for the multiplicative Hilbert matrix; see Corollary~\ref{cor:bddsymbol} for the precise meaning of this latter statement. The problem was then later restated in the survey \cite[Problem~3.2]{SaksmanSeipSurvey}, and has remained open.

Some weaker positive results were obtained, however.  In \cite{BrevigBayart} it was proven that
\[
\left|\sum_{n=2}^\infty \frac{a_n}{\sqrt n\log n}\right| \lesssim_p \|f\|_{\HH^p}, \qquad p > 1,
\]
and in \cite{BondarenkoEtAl} it was shown that
\[
\left|\sum_{n=2}^\infty \frac{a_n}{\sqrt n (\log n)^\beta}\right| \lesssim_\beta \|f\|_{\HH^1}, \qquad \beta > 1.
\]

Around the same time, Harper's work on multiplicative chaos \cite{Harper} showed that Carleson embeddings really are $p$-dependent. A basic feature of the $\HH^2$-theory is the so-called local embedding inequality \cite{HLS},
\[
\int_0^1 |f(1/2 + it)|^2 \, dt \lesssim \|f\|_{\HH^2}^2.
\]
For a long time, it was an open problem whether the local embedding inequality held for non-even values  of $q$,
\begin{equation} \label{eq:qemb}
\int_0^1 |f(1/2 + it)|^q \, dt \lesssim \|f\|_{\HH^q}^q?
\end{equation}
The investigations of \cite{Harper} showed, as a by-product, that this local embedding inequality is \textit{false} for $1 \leq q < 2$. The inequality is still open for $q > 2$, $q \notin 2\mathbb{N}$. Note that if \eqref{eq:qemb} had been true for $q = 1$, it would have implied the Carleson embedding inequality
\begin{equation} \label{eq:Carlesonline}
\int_0^1 |f(1/2 + \sigma)| \, d\sigma \lesssim \|f\|_{\HH^1},
\end{equation}
stronger than \eqref{eq:signedhardy}, see \cite[Theorem 4]{OlsenSaksman}.

The main result of this paper is that the embedding inequality \eqref{eq:Carlesonline} does hold. Before stating the theorem, let us first briefly explain the standard notation in the study of Hardy spaces of Dirichlet series. We associate each $z \in \TT^\infty$ with a completely multiplicative character $z \colon \NN \to \TT$, uniquely determined by the fact that $z(p_j) = z_j$, where $p_j$ is the $j$th prime, $j \geq 1$. In this way, we identify each Dirichlet polynomial
\[
 f(s)=\sum_{n=1}^N a_n n^{-s}
 \]
with its Bohr lift  $\mathscr{B} f$, a polynomial in the variables $z = (z_1, z_2, \ldots)$,
\[
 \mathscr{B} f(z) = F(z) = \sum_{n=1}^N a_n z (n), 
 \qquad z \in \TT^\infty.
\]
The Hardy space $\HH^p$-norm of a Dirichlet polynomial $f$, $1 \leq p < \infty$, is then given by
\[
 \|f\|_{\HH^p}^p := \int_{\TT^\infty}|F(z)|^p\,dm(z) = \lim_{T\to\infty} \frac{1}{2T} \int_{-T}^T |f(it)|^p \, dt 
\]
where $m$ is the normalized Haar measure on $\TT^\infty$, and the second equality is a standard consequence of Kronecker's theorem. The space $\HH^p$, $1 \leq p < \infty$,  is defined by the completion of Dirichlet polynomials in this norm. See \cite{QueffelecQueffelec} for further details. 

\begin{theorem}\label{thm:main}
There is a constant $C_0 >0$ such that 
\begin{equation}\label{eq:Hardyabsvalue}
 \int_0^1 |f(1/2+\sigma)|\,d\sigma\le C_0\|f\|_{\HH^1}, \qquad f \in \HH^1.
\end{equation}
Therefore, there is a constant $C_0' > 0$ such that
\begin{equation}\label{eq:hardy}
 \left|a_1+\sum_{n=2}^\infty \frac{a_n}{\sqrt n\log n}\right|
 \le C_0' \|f\|_{\HH^1}, \qquad f \in \HH^1.
\end{equation}
The sum in \eqref{eq:hardy} is convergent for every $f \in \HH^1$.
\end{theorem}

To explain that the multiplicative Hilbert matrix is generated by a bounded symbol, let $P_+ \colon L^2(\TT^\infty) \to \mathscr{B}(\HH^2)$ be the orthogonal projection. That is, for $F \in L^2(\TT^\infty)$ with Fourier series $F(z) = \sum_{q \in \mathbb{Q}_+} a_q z(q)$, where characters $z \in \TT^\infty$ now act completely multiplicatively on the positive rationals $\mathbb{Q}_+$, we have that
$$P_+ F (z) = \sum_{n \in \NN} a_n z(n).$$
Inequality \eqref{eq:hardy} can be restated by saying that the Hilbert matrix symbol $\phi(s) = 1 + \sum_{n=2}^\infty \frac{1}{\sqrt{n} \log n} n^{-s}$ induces a bounded functional on $\HH^1$. Since $\mathscr{B} \HH^1 \subset L^1(\TT^\infty)$, an appeal to Hahn--Banach therefore yields the following corollary. 
\begin{corollary} \label{cor:bddsymbol}
The multiplicative Hilbert matrix $\mathbf{H}$ has a bounded symbol. That is, there is $B\in L^\infty(\TT^\infty)$ such that
\[
 P_+B(z) = 1 + \sum_{n=2}^\infty\frac{z(n)}{\sqrt n\log n}.
\]
\end{corollary}

Before embarking on the proof of Theorem~\ref{thm:main}, let us note that the validity of the \textit{Hardy inequality}, featuring absolute values on the coefficients,
\begin{equation} \label{eq:Hardyineqreal}
 |a_1|+\sum_{n\ge2}\frac{|a_n|}{\sqrt n\log n}
 \lesssim\|f\|_{\HH^1}
\end{equation}
remains an open problem. In one variable, the Hardy inequality
\[
\sum_{n=0}^\infty \frac{|\hat{g}(n)|}{n+1} \leq \pi\|g\|_{H^1(\mathbb{T})}
\]
is a relatively direct consequence of the boundedness of the Hilbert matrix, since we can factor any $H^1(\mathbb{T})$-function into a product of two $H^2(\mathbb{T})$-functions.

\section{Proof of Theorem~\ref{thm:main}}
\label{sec:proof}
By a straightforward argument, it will be sufficient to show that there is a constant $C_0 > 0$ such that
\begin{equation} \label{eq:mainpointineq}
\sum_{j=0}^\infty e^{-j-\theta}
      |f(1/2+e^{-j-\theta})|\leq C_0 \|f\|_{\HH^1}, \qquad \theta \in (0,1), \; f \in \HH^1.
\end{equation}
To unburden the notation, we let
\[
\sigma_j = e^{-j-\theta},
\]
suppressing the dependence on $\theta$ both in this and all future quantities and sets to be considered. 

We fix a Dirichlet polynomial $f = \sum_{n=1}^N a_n n^{-s}$ and an integer \(J\ge0\) for the duration of the proof. It is important to keep $f$ and $J$ fixed, but arbitrary, as the remainder of the constructions will depend on them. As with $\theta$, we suppress this dependence from the notation. 

We will achieve \eqref{eq:mainpointineq} by constructing suitable functions $\psi_j \in L^\infty(\TT^\infty)$, $\|\psi_j\|_\infty \lesssim \sigma_j^{-1}$, which reproduce $f$ at $1/2 + \sigma_j$,
$$\int_{\TT^\infty} \mathscr{B}f(z) \overline{\psi_j(z)} \, dm(z) = f(1/2+\sigma_j).$$
Crucially, we will ensure that the supports $E_j = \supp \psi_j \subset \TT^\infty$ of these reproducing functions overlap boundedly; each set $E_j$ will only intersect a fixed number of other sets $E_k$.  The precise choice of functions $\psi_j$, as well as their supports $E_j$, will actually depend on the specific function $f$ considered, as well as $J$ and other parameters, introduced later.

To give the general recipe, we first work on the finite-dimensional torus $\TT^d$, $d<\infty$. We denote coordinate-wise real positive points in $\DD^d$ by
\[
r=(r_\ell)_{\ell=1}^d\in(0,1)^d.
\]
The standard analytic reproducing kernel of $H^2(\TT^d)$ at $r$ is given by 
\[ K_r(z) =\prod_{\ell=1}^d(1-r_\ell z_\ell)^{-1}, \qquad z \in \TT^d, \]
and has norm given by
\[\|K_r\|^2_{H^2(\TT^d)} = \prod_{\ell=1}^d(1-r_\ell^2)^{-1}.\]
We denote
\[
\Phi_r(z) := \|K_r\|^2_2 \frac{K_r(z)}{\overline{K_r(z)}}, \qquad z \in \TT^d.
\]
Then $\Phi_r \in L^\infty(\TT^d)$,  $|\Phi_r| \equiv \|K_r\|^2_2$, is one canonical representative of point evaluation at $r$, that is,
\[
P_+ \Phi_r = K_r.
\]
In the dimensional limit $d \to \infty$, this proves that every $r \in (0,1)^\infty \cap \ell^2$ induces a bounded point evaluation on $H^1(\TT^\infty) = \mathscr{B}(\HH^1)$, see \cite{ColeGamelin1986}. 

The key insight to our construction is that it is actually possible to restrict $\Phi_r$ to certain relatively small subsets of $\TT^d$, while maintaining the reproducing property. To formulate the basic building block precisely, we let $T_r \colon \TT^d \to \TT^d$ denote the coordinate-wise M\"obius map which as a map $T_r \colon \DD^d \to \DD^d$ sends $0$ to $r$,
\[
 T_r(z) = \left( \frac{r_1+z_1}{1+r_1 z_1}, \ldots, \frac{r_d+z_d}{1+r_d z_d} \right).
\]
\begin{lemma}\label{lem:localization}
Suppose that \(D\subset\TT^d\) has positive measure $0 < m(D) \leq 1$, and is invariant under diagonal rotations,
\[
 e^{i\vartheta}D=D, \qquad \vartheta\in\mathbb R.
\]
Let 
\begin{equation}\label{eq:localizedsymbol}
 \psi_{r,D}(z)=
 \Phi_r(z)\frac{\ind_{T_r(D)}(z)}{m(D)}, \qquad z \in \TT^d.
\end{equation}
Then 
\begin{equation}\label{eq:localizedproperties}
 P_+\psi_{r,D}=K_r,
 \qquad
 |\psi_{r,D}|=\frac{ \|K_r\|^2_2}{m(D)}\ind_{T_r(D)}.
\end{equation}
\end{lemma}

\begin{proof}
Under the change of variable $z = T_r(w)$, $w \in D$, a calculation shows that
\[
 dm(z)= \prod_{\ell=1}^d
       \frac{1-r_\ell^2}{|1+r_\ell w_\ell|^2} \, dm(w)
\]
and that
\[
 K_r(T_r(w))= \|K_r\|^2_2\prod_{\ell=1}^d(1+r_\ell w_\ell).
\]
Thus, for any analytic polynomial $F$,
\[
 \int_{\TT^d}F\overline{\psi_{r,D}}\,dm= \frac{1}{m(D)
} \int_D H \,dm,
 \qquad
 H(w)=F(T_r(w))\prod_{\ell=1}^d(1+r_\ell w_\ell)^{-2}.
\]
This function $H$ extends analytically to a function on $\overline{\DD^d}$ with
\(H(0)=F(r)\). Therefore,
\[
 \frac1{2\pi}\int_0^{2\pi}H(e^{i\vartheta}w)\,d\vartheta=H(0) = F(r), \qquad w \in \TT^d.
\]
By invariance of the set $D$, we therefore have that
\[
  \int_{\TT^d}F\overline{\psi_{r,D}}\,dm = \frac{1}{m(D)
} \int_D H\,dm
 = \frac{1}{m(D)
} \int_D\frac1{2\pi}\int_0^{2\pi}H(e^{i\vartheta}w)\,d\vartheta\,dm(w)
 = F(r).
\]
Since $\psi_{r,D} \in L^\infty(\TT^d)$, we have that $P_+ \psi_{r,D} \in H^2(\TT^d)$, and the preceding calculation shows that for any analytic polynomial $F$,
\[
  \int_{\TT^d}F\overline{P_+ \psi_{r,D}}\,dm =   \int_{\TT^d}F\overline{\psi_{r,D}}\,dm
 = F(r).
\]
We conclude that $P_+ \psi_{r,D} = K_r$, as promised.
\end{proof}

\begin{remark}
The change of measure under $T_r$ is a realization of what Harper \cite[Section~3.2]{Harper} refers to as the introduction of a "tilted probability measure". In our case, imposing diagonal rotational invariance on the set $D$ is what allows us to restrict the kernel to $T_r(D)$, yet preserve the reproducing property of the kernel. The sets $D$ will be chosen in connection with a probabilistic argument, see Definition~\ref{def:Dj}. 
\end{remark}

Let $\delta=e^{-(J+1)}$, and choose $P$, depending on $f$ and $J$, such that $n \mid P$ for every $n \leq N$, and so large that the sum over primes
\[
 \sum_{p>P}p^{-1-2\delta}\le1.
\]
We now switch our point of view slightly, and index points $r \in (0,1)^d$ over the primes $p \leq P$, where $d$ is the number of such primes. More precisely, we consider the $J+1$ points $r^j \in (0,1)^d$, $0 \leq j \leq J$, given by
\[
 r_p^j=p^{-1/2-\sigma_j}, \qquad p \leq P.
\]
Recall that $\sigma_j=e^{-j-\theta}$, for an arbitrary $\theta\in(0,1)$.

For $0 \leq j,k \leq J$, let $a^{jk}$ denote the transition point $a^{jk}= T_{r^k}^{-1}(r^j)$, so that 
\[
T_{a^{jk}} = T_{r^k}^{-1}\circ T_{r^j}.
\]
In order to be able to control the intersections $T_{r^j}(D_j) \cap T_{r^k}(D_k)$, we will specify the sets $D_j$ in terms of the "energy"
\[
\mathcal{E}_{jk}= \|a^{jk}\|_{\ell^2}^2 = \sum_{p\leq P}(a_p^{jk})^2.
\]
The following lower bound is precisely the input needed for the probabilistic part of the argument. Its proof is routine.
\begin{lemma}\label{lem:arithmetic}	
It holds that $ |a_p^{jk}| \leq p^{-1/2}$ and
\begin{equation*}
 \mathcal{E}_{jk}\ge |j-k|-8
\end{equation*}
for all $0 \leq j,k \leq J$ and $p \leq P$.
\end{lemma}

\begin{proof}
For \(0\le y\le x<1\),
\[
 x-y\le\frac{x-y}{1-xy}\le x.
\]
Since $a_p^{jk} = (T_{r^k}^{-1}(r^j))_p =  (r_p^j-r_p^k)/(1-r_p^jr_p^k)$, this proves the first inequality, and shows that
$$\mathcal{E}_{jk} \geq \sum_{p \leq P} |r_p^j-r_p^k|^2.$$

To prove the statement from here, we really only need to know that the prime zeta function \(\mathcal P(s)=\sum_p p^{-s}\) satisfies $\mathcal{P}(1+\sigma) = -\log \sigma + O(1)$ when $\sigma \to 0^+$. More precisely, note, by expansion of the zeta Euler product, that
\[
 \log\zeta(1+\sigma)=\mathcal P(1+\sigma)+R_1(\sigma), \qquad \sigma > 0,
\]
with an error $R_1$ satisfying $0 \leq R_1(\sigma) \leq 1/2$. Since, by integral comparison,
\[
 \frac{1}{\sigma}\le\zeta(1+\sigma)\le1+\frac{1}{\sigma}, \qquad \sigma > 0,
\]
we find that 
\begin{equation*}
 \mathcal P(1+\sigma)= - \log \sigma +R_2(\sigma), 
\end{equation*}
where $-\frac{1}{2} \le R_2(\sigma) \le \log 3 < 2$ whenever $0 < \sigma \leq 2$.

For \(u,v\in(0,1]\), we therefore have
\begin{align*}
 \sum_p(p^{-1/2-u}-p^{-1/2-v})^2
 &=\mathcal P(1+2u)+\mathcal P(1+2v)-2\mathcal P(1+u+v)\\
 &\ge \log\frac{(u+v)^2}{4uv}-5\\
 &\ge |\log u-\log v|-2\log2- 5 > |\log u-\log v|- 7 
\end{align*}
By assumption, for $u, v \geq \delta$, we have that
\[
  \sum_{p > P} (p^{-1/2-u}-p^{-1/2-v})^2 \leq \sum_{p>P}p^{-1-2\delta}\le1.
\]
We apply this with \(u=\sigma_j\), \(v=\sigma_k\), to conclude that
\[
 \mathcal{E}_{jk}\geq \sum_{p \leq P} |r_p^j-r_p^k|^2 \geq  |j-k| - 8. \qedhere
\]
\end{proof}

We now define the sets $D_j$ which will serve as the input to Lemma~\ref{lem:localization}. Here and in the remainder of the proof, we fix a very large number $L$ and a small number $\varepsilon$; to be concrete, let
\[
L = 10^5 \; \textrm{ and } \; \varepsilon = \frac{1}{32}. 
\]

\begin{definition} \label{def:Dj}
For $0 \leq j \leq J$, let $D_j\subset\TT^d$ be the set of $w$ such that
\begin{equation*}\label{eq:conditions}
 \left|\sum_{p\le P}a_p^{jk}w_p\right|\le\varepsilon \mathcal{E}_{jk} \; \textrm{ and } \;
 \left|\sum_{p\le P}(a_p^{jk}w_p)^2\right|\le\frac12\mathcal{E}_{jk}, \quad \forall k \leq J \textrm{ s.t. } |k-j|\geq L.
\end{equation*}
\end{definition}
Note that each set $D_j$ is invariant under diagonal rotation, so that we may apply Lemma~\ref{lem:localization}. Furthermore, the definition is such that basic concentration inequalities ensure that $D_j$ always has large measure, which is necessary in order to control the denominator in \eqref{eq:localizedsymbol}.

\begin{lemma}\label{lem:probability}
For every $0 \leq j \leq J$,
\[
 m(D_j)\geq \frac{1}{2}.
\]
\end{lemma}

\begin{proof}
For independent Steinhaus variables $w_\ell$ (uniformly distributed on $\TT$) and real coefficients
\(b_\ell\), the Hoeffding inequality, applied separately to the real and imaginary parts, yields that 
\[
 \PP\left(\left|\sum_\ell b_\ell w_\ell\right|>x\right)
 \le4\exp\left(-\frac{x^2}{4 \sum b_\ell^2}\right).
\]
For each pair \(j,k\), $0 \leq j, k \leq J$, we conclude, recalling $\mathcal{E}_{jk} = \|a^{jk}\|_{\ell^2}^2$, that
\[
m \bigl( \{ w \, : \,  \bigl|\sum_{p\le P}a_p^{jk}w_p \bigr| > \varepsilon \mathcal{E}_{jk} \} \bigr) \leq  4 \exp(-  \varepsilon^2 \mathcal{E}_{jk}/4) = 4 \exp(- \mathcal{E}_{jk}/4096) .
\]
The variables $w_p^2$ can again be understood as independent Steinhaus variables, and after noting that
\[
 \sum_p(a_p^{jk})^4\le\sum_p(a_p^{jk})^2=\mathcal{E}_{jk},
\]
a second application of Hoeffding's inequality gives
\[
m \bigl( \{ w \, : \,  \bigl|\sum_{p\le P}(a_p^{jk}w_p)^2\bigr| > \frac{1}{2} \mathcal{E}_{jk} \} \bigr) \leq 4 \exp(- \mathcal{E}_{jk}/16) .
\]
Applying the definition of $D_j$, and then Lemma~\ref{lem:arithmetic}, we conclude that
\[
 m(\TT^d \setminus D_j)
 \leq 16e^{8/4096}\sum_{h\ge L}e^{-h/4096} < \frac{1}{2}. \qedhere
\]
\end{proof}

On the other hand, the choice of $D_j$ ensures that the sets 
\[
E_j := T_{r^j}(D_j) = \supp \psi_{r^j, D_j}
\]
have uniformly finite overlap. The point is that if $z \in E_j \cap E_k \neq \emptyset$ for two indices with $|j-k| \geq L$, the linear conditions in Definition~\ref{def:Dj} force the corresponding linear sums to be small at both $T_{r^j}^{-1}z$ and $T_{r^k}^{-1}z$. However, the quadratic condition forces the transition map $T_{a^{jk}}$ to impose a larger change to any such sum, which is contradictory. In light of the concentration inequalities used in the proof of Lemma~\ref{lem:probability}, this argument suggests an interpretation in terms of linear and quadratic statistics. \begin{lemma}\label{lem:overlap}
	We have that $E_j \cap E_k = \emptyset$ whenever $|j - k| \geq L$.
\end{lemma}

\begin{proof}
It suffices to prove that if $z\in T_{r^j}(D_j) \cap T_{r^k}(D_k)$, then $|j - k| < L$. So suppose for contradiction that $|j-k|\ge L$. Let
\[
 w=T_{r^j}^{-1}z,\qquad w'=T_{r^k}^{-1}z,\qquad
 a=a^{jk},\qquad \mathcal{E}=\mathcal{E}_{jk}.
\]
The transition point $a$ was defined so that
\[
 w'_p=T_{a_p}(w_p)=\frac{w_p+a_p}{1+a_pw_p}.
\]
By calculation, we therefore have that 
\[
 a_p\,\Re\bigl(w'_p-w_p\bigr)
 =\frac{a_p^2(1-\Re w_p^2)}{|1+a_pw_p|^2}
 \ge\frac{1}{4}a_p^2(1-\Re w_p^2).
\]
The second defining condition of the set $D_j$ therefore gives, since $w \in D_j$, that
\[
 \Re\sum_pa_p(w'_p-w_p)
 \ge\frac14\left(\mathcal{E}-\Re\sum_pa_p^2w_p^2\right)
 \ge\frac18 \mathcal{E},
\]
while the first defining condition, noting also that $w' \in D_k$ and that $a^{kj} = -a^{jk}$, yields
\[
 \left|\sum_pa_p(w'_p-w_p)\right|
 \le2\varepsilon \mathcal{E}=\frac1{16} \mathcal{E}.
\]
These two estimates contradict each other.
\end{proof}

We can now finish the proof of Theorem~\ref{thm:main}. We apply Lemma~\ref{lem:localization} with $r = r^j$ and $D=D_j$, giving us a symbol for point evaluation at $r^j$,
\[
 \psi_{j} := \psi_{r^j,D_j}
 =\Phi_{r^j} \frac{\ind_{T_{r^j}(D_j)}}{m(D_j)}.
\]
Noting that
\[
 \sigma_j|\Phi_{r^j}| \leq \sigma_j \zeta(1+2\sigma_j) \leq \sigma_j+\frac{1}{2} \leq \frac{3}{2},
\]
Lemmas~\ref{lem:probability} and \ref{lem:overlap} imply the pointwise
estimate
\[
 \sum_{j=0}^J \sigma_j| \psi_{j}(z)|
 \leq 3L.
\]

Let $F = \mathscr{B} f$ be the Bohr lift of the Dirichlet polynomial $f$ that we fixed at the outset of the proof.  Then, by the reproducing property of $\psi_{j}$,
\[
 \sum_{j=0}^J \sigma_j|F(r^j)|
 \leq \int_{\TT^d} |F(z)|\sum_{j=0}^J \sigma_j| \psi_{j}(z)|\,d m(z) \leq 3L\|F\|_{L^1(\TT^d)}.
 \]
Recalling that $F(r^j)=f(1/2+\sigma_j) = f(1/2 + e^{-j-\theta})$, and integrating over $\theta \in (0,1)$, we conclude that
\[
 \int_\delta^1 |f(1/2+\sigma)|\,d\sigma
 =\int_0^1\sum_{j=0}^J e^{-j-\theta}
      |f(1/2+e^{-j-\theta})|\,d\theta \leq 3L \|f\|_{\HH^1}.
\]
Since $\delta = e^{-(J+1)}$, running the construction with larger and larger $J \to \infty$ establishes \eqref{eq:Hardyabsvalue} for the chosen Dirichlet polynomial $f$. By density, the inequality then extends to all $f \in \HH^1$, proving \eqref{eq:Hardyabsvalue}.

When $f$ is a Dirichlet polynomial, the inequality \eqref{eq:hardy} is a direct consequence, upon observing the identity
\[
 a_1+\sum_{n=2}^N\frac{a_n}{\sqrt n\log n}
 =\int_0^1 f(1/2+\sigma)\,d\sigma+
   \int_1^\infty(f(1/2+\sigma)-a_1)\,d\sigma
\]
and the trivial estimate
\[
\int_1^\infty |f(1/2+\sigma)-a_1|\,d\sigma \leq \sum_{n=2}^N\frac{|a_n|}{n^{3/2}\log n} \leq \|f\|_{\HH^1}\sum_{n=2}^N\frac{1}{n^{3/2}\log n}.
\]

It only remains to justify the convergence of $\sum_{n=2}^\infty \frac{a_n}{\sqrt n\log n}$ for a general $f \in \mathscr{H}^1$. For $t > 0$, let
\[
G(t) = \int_{t}^\infty (f(1/2 + \sigma) - a_1) \,d\sigma = \sum_{n=2}^\infty \frac{a_n}{\sqrt{n} \log n} n^{-t}.
\]
The already established inequality \eqref{eq:Hardyabsvalue} implies the existence of $G(0) = \lim_{t \to 0^+} G(t)$. Moreover, we can use the Helson inequality \cite{Helson}, 
\[
\sum_{n=2}^\infty \frac{|a_n|^2}{d(n)} \leq \|f\|_{\HH^1}^2,
\]
where $d(n)$ denotes the divisor function. A standard argument, almost identical to the proof of the classical Fejér Tauberian theorem \cite[Section~7]{Zalcman}, then shows that
\[
\lim_{N \to \infty} \left( \sum_{n=2}^N  \frac{a_n}{\sqrt n\log n} - G((\log N)^{-1}) \right) = 0,
\]
completing the proof.

\subsection*{Acknowledgments.}
The author was supported by grant no. 334466 of the Research Council of Norway, "Fourier Methods and Multiplicative Analysis". The proof was developed with assistance from GPT 5.6 Sol. It initially grew out of an attempt to disprove \eqref{eq:Hardyineqreal}, which necessitated an (AI-driven) study of multiplicative chaos. Setting the direction to use Lemma~\ref{lem:localization} in this context resulted in a first draft of the proof.  The proof was then distilled, corrected in minor ways, simplified,  verified, and written in its final form by the author.

\end{document}